\RequirePackage{fix-cm}
\documentclass[11pt,reqno]{amsart}

\usepackage{hyperref}
\usepackage{color}
\usepackage{graphicx}
\usepackage{latexsym}
\usepackage{enumerate}
\usepackage{cite}
\usepackage{paralist}
\usepackage{tikz}
\usetikzlibrary{decorations.pathmorphing}
\usetikzlibrary{decorations.markings}
\usetikzlibrary{shapes,positioning,matrix,arrows}
\usetikzlibrary{patterns}
\usepackage{enumerate}
\usepackage{amsmath,amsfonts}
\usepackage{bm}
\usepackage[margin=1in]{geometry}
\usepackage{mathtools}

\newcommand{\Z}{\hbox{$\mathbb Z$}}

\newcommand{\R}{\hbox{$\mathbb R$}}

\DeclareMathOperator{\sgn}{sgn}

\DeclareMathOperator{\cone}{cone}
\DeclareMathOperator{\conv}{conv}
\DeclareMathOperator{\comp}{Comp}
\DeclareMathOperator{\bicomp}{BiComp}

\DeclareMathOperator{\coloneqq}{:=}
\newcommand{\PG}{\mathcal{P}_G}

\theoremstyle{definition}
\newtheorem{construction}{Construction}[section]
\newtheorem{observation}[construction]{Observation}

\newtheorem{definition}[construction]{Definition}
\newtheorem{proposition}[construction]{Proposition}
\newtheorem{example}[construction]{Example}
\newtheorem{lemma}[construction]{Lemma}
\newtheorem{remark}[construction]{Remark}
\newtheorem{theorem}[construction]{Theorem}
\newtheorem{corollary}[construction]{Corollary}

\newtheorem*{spthmA}{Propositions \ref{prop:signedConeFacets} and \ref{prop:mixedConeFacets}}

\newtheorem*{spthmB}{Theorems \ref{thm:signedSerres} and \ref{thm:mixedSerres}}

\title{Serre's Properties for Quadratic Generated Domains from Graphs
}

\author{Drew J. Lipman}
\address{Dept of Mathematical Sciences, Clemson University}
              \email{djlipma@clemson.edu}

\author{Michael A. Burr}
\address{Dept of Mathematical Sciences, Clemson University}
\email{burr2@clemson.edu}
\thanks{Partially supported by a grant from the Simons Foundation (\#282399 to Michael Burr) and NSF Grant CCF-1527193.}

\date{\today}

\begin{document}
\begin{abstract}
For any graph, one can construct a ring, called the edge ring, which is a quadratic-monomial generated subring of the Laurent polynomial ring $k[x_1^{\pm 1},\dots,x_n^{\pm 1}]$.  In fact, every quadratic\\ -monomial generated subring of this Laurent polynomial ring can be generated as an edge ring for some graph.  The combinatorial structure of the graph has been successfully applied to identify and classify many important commutative algebraic properties of the corresponding edge ring.  In this paper, we classify Serre's $R_1$ condition for all quadratic-monomial generated subrings of $k[x_1^{\pm 1},\dots,x_n^{\pm 1}]$.  Moreover, we provide a minimal example of a graph whose corresponding edge ring is not Cohen-Macaulay.  This paper extends the work of Hibi and Ohsugi from the setting subrings of polynomial rings to subrings of Laurent polynomial rings.
\end{abstract}

\maketitle

\keywordsname:~Edge Rings, Serre's $R_1$ Condition, Cohen-Macaulay, Quadratic-monomial Generated Domains

\section{Introduction}
\label{sec:Intro}
Hibi and Ohsugi \cite{OhsugiHibi:1998} and Simis, Vasconcelos, and Villarreal \cite{Villarreal:1998} independently give a construction of a ring $k[G]$, called the {\em edge ring}\footnote{The edge ring is a different object than the edge ideal of a graph.  The edge ring and edge ideal are generated by the same elements, but one as a subring and the other as an ideal of $k[x_{1},\dots,x_n]$.  For more details on the edge ideal, see, for example, \cite{Villarreal:1994,Sullivant:2008} and the references included therein.  It is also distinct from the edge algebra (sometimes also called the edge ring), see, for example \cite{Estrada:2000,GitlerValencia:2005} and the references included therein.  There is another distinct concept called the edge ring, see, for example \cite{Zalavsky:2009}.}, from a graph $G$.  In their work, $k[G]$ is a quadratic-monomial generated subring of the polynomial ring $k[x_1,\dots,x_n]$ over a field $k$, where each edge of $G$ corresponds to a generator of $k[G]$.  In \cite{LipmanBurr:2016}, we generalized this construction to all quadratic-monomial generated subrings of the the Laurent polynomial ring $k[x_1^{\pm 1},\dots,x_n^{\pm 1}]$.

Since the edge ring is constructed from combinatorial data, the the edge ring has proved to be a fruitful construction for studying commutative algebraic properties, see, e.g., \cite{HibiMatsudaOhsugi:2016,HibiNishiyamaOhsugiShikama:2014,HibiKatthan:2014,HibiHigashitaniKimura:2014,TatakisThoma:2013,Matsui:2003,GitlerVillareal:2011,HerzogHibiZheng:2004,OhsugiHibi:2000,DAli:2015,HibiOhsugi:2006,Villarreal:2005,HibiMoriOhsugiShikama:2016,BermejoGarciaReyes:2015,BermejoGimenezSimis:2009,OhsugiHibi:2017,HibiKattan:2014,OhsugiHibi:1998,Villarreal:1998,LipmanBurr:2016}.  These papers use the interplay between the combinatorics and commutative algebra to classify commutative algebraic properties for edge rings and use edge rings to easily construct rings which are examples and non-examples for these properties.  For example, in \cite{OhsugiHibi:1998} and \cite{Villarreal:1998}, the authors gave a combinatorial characterization, called the odd cycle condition, of the normality of $k[G]$ in terms of $G$.  Moreover, when $k[G]$ is not a normal domain, they use the combinatorial data of $G$ to construct the normalization of $k[G]$.  In \cite{LipmanBurr:2016}, we extended their work to completely classify the normality of all quadratic-monomial generated subrings of the Laurent polynomial ring.  The current paper continues this generalization by extending the work in \cite{HibiKattan:2014} to all quadratic-monomial generated Laurent polynomial rings.

In commutative algebra, the $R_\ell$ and $S_\ell$ conditions for a ring characterize many important properties, such as normality and Cohen-Macaulayness, see Section \ref{sec:SerreConditions} and \cite{Grothendieck:1965,BrunsHerzog-CohenMacaulay,Serre:2000} for details.  In \cite{Vitulli:2009}, Vitulli characterizes the $R_\ell$ condition for semigroup rings, and, in \cite{HibiKattan:2014}, Hibi and Katth\"an use this characterization to characterize the $R_1$ condition for edge rings.  In this paper, we generalize their work and provide a complete characterization of the $R_1$ condition for all quadratic-monomial generated domains in the Laurent polynomial ring $k[x_1^{\pm 1},\dots,x_n^{\pm 1}]$.  This case is considerably more complicated than the situation in the previous work because the negative powers allow exponents to cancel.  In addition, the conditions in \cite{HibiKattan:2014} do not appear well-suited to generalizations; therefore, we reinterpret and these conditions in a way that can be suitably generalized.  In order to address these difficulties, we introduce new proofs; in particular, our proofs are more geometric and combinatorial in nature than in the previous work.  In addition, our generalization provides new examples exhibiting a graph whose edge ring satisfies $R_1$ and fails $S_2$.

\subsection{Main Results}

Suppose that $R$ is a quadratic-monomial generated subring of $k[x_1^{\pm 1},\dots,x_n^{\pm 1}]$.  From $R$, we construct a mixed signed, directed graph $G$, i.e., the edges of $G$ are either directed or signed such that the edge ring $k[G]$ associated to $G$ equals $R$.  For each edge $e\in G$, assign a point $\rho(e)\in\mathbb{R}^n$ and define $\PG$ to be the convex hull of these points.  Our first result characterizes the subgraphs of $G$ which correspond to facets of $\cone(\PG)$.  In particular, in Observation \ref{thm:dimFormula} and Proposition \ref{prop:mixedConeDimension}, we prove that the codimension of $\cone(\PG)$ is the number of bipartite components of $G$.

Next, we define facet subgraphs in Definitions \ref{def:facetsubgraphdirected} and \ref{def:facetsubgraphsigned} and prove that 

\begin{spthmA}
Let $G$ be a signed graph.  $F$ is a facet of $\cone(\PG)$ if and only if there exists a facet subgraph $H$ such that $\cone(\mathcal{P}_H)=F$.
\end{spthmA}

Next, we use \cite[Proposition~3.2]{HibiKattan:2014}, see Proposition \ref{prop:vitulli1} for a special case, to characterize whether $R$ satisfies $R_1$ in terms of the facets of the cone of $\PG$.  Using this result, we develop a combinatorial condition on whether $k[G]$ satisfies $R_1$.

\begin{spthmB}
Let $G$ be a graph.
$k[G]$ satisfies Serre's $R_1$ condition if and only if for every facet subgraph $H$ of $G$, $H$ has at most one more component than $G$.
\end{spthmB}

We use this characterization, along with the condition of normality from \cite{LipmanBurr:2016}, in order to provide an example of a graph whose edge ring satisfies $R_1$, but not $S_2$.  This provides an example of an edge ring which is not Cohen-Macaulay.

\subsection{Outline of Paper}

The remainder of this paper is organized as follows:  In Section \ref{sec:BackAndNote}, we provide the background and notation used throughout the paper.  In Section \ref{sec:Geometry}, we compute the dimension of the cone $\cone(\PG)$ for the signed graph $G$, define facet subgraphs, and classify which graphs $G$ satisfy Serre's $R_1$ condition in terms of combinatorial data.  In Section \ref{sec:mixedSerres}, we extend the results of Sections \ref{sec:Geometry} to all quadratic-monomial generated domains.  Finally, we conclude in Section \ref{sec:Conclusion}.

\section{Background and Notation}
\label{sec:BackAndNote}
In this section, we recall notation, definitions, and results from graph theory, semigroup theory, and Serre's conditions for use in this paper.  Our notation for edge rings follows the notation of Hibi and Ohsugi \cite{OhsugiHibi:1998}.

\subsection{Graph Theory}
\label{sec:GT}

One main object of study in this paper is a signed graph.
We give the basic definitions for a signed graph in this section.

\begin{definition}
A {\em signed graph} $(G,\sgn)$ is an undirected graph $G=(V,E)$ and a {\em sign function} $\sgn:E\rightarrow \{-1,+1\}$ where $\sgn(e)$ denoted the {\em sign} of the edge $e\in E$.  For notational convenience, an edge $ij$ with $\sgn(ij)=+1$ or $-1$ is denoted $+ij$ or $-ij$, respectively.
We omit the sign when it is understood from context.
\end{definition}

\begin{definition}
Let $G$ be a signed graph.
$H$ is a {\em component} of $G$ if it is a maximal connected subgraph.
We say a component $H$ is a {\em bipartite component} if the vertices of $H$ can be partitioned into two sets $L$ and $R$ so that all the edges of $H$ have exactly one vertex in each of $L$ and $R$.
\end{definition}

The construction of a ring from a graph proceeds by first constructing a semigroup from the graph.  We now define the map used in this construction.

\begin{definition}[cf {\cite{OhsugiHibi:1998}}]\label{def:edgepolytope}
Let $G$ be a signed graph with $n$ vertices, possibly with loops, and without multiple edges.
Define a map $\rho:E(G)\rightarrow \R^n$ as $\rho(e)=\sgn(e)(e_i + e_j)$ where $e=+ij$ or $e=-ij$ is an edge of the graph.
When $e$ is a loop, $i=j$ and $\rho(ii)=2\sgn(ii)e_i$.
Let $\rho(E(G))$ be the image of $E(G)$ and define the {\em edge polytope of $G$} as $\mathcal{P}_G\coloneqq\conv(\rho(E(G)))$.
\end{definition}

Observe that in \cite{OhsugiHibi:1998}, the authors do not consider signed graphs, and, hence, all edges $e$ in $G$ have positive sign, i.e., $\rho(e) = e_i + e_j$.

\subsection{Semigroups}\label{sec:SemiGPs}

The edge ring is constructed as a semigroup ring.  In this section, we recall the definition of affine semigroups and the edge ring.  For more details, see, e.g., \cite{BrunsHerzog-CohenMacaulay,CoxLittleSchenck-ToricVarieties}.

\begin{definition}
An {\em affine semigroup} $C$ is a finitely generated semigroup containing zero which, for some $n$, is isomorphic to a subsemigroup of $\Z^n$.
Over a field $k$, the {\em affine semigroup ring} $k[C]$ of $C$ is the ring generated by the elements of $\{x_1^{c_1}\dots x_n^{c_n}:(c_1,\dots,c_n)\in C\}$.  We often use multi-index notation for clarify, where $x^c=x_1^{c_1}\dots x_n^{c_n}$ with $c=(c_1,\dots,c_n)$.
\end{definition}

For a set $S\subseteq\Z^n$, we observe the set $\Z_+S$, i.e., the set of all positive integral linear combinations of the elements of $S$, is the smallest subsemigroup of $\Z^n$ containing $S$.  We define the cone of $S$ to be all positive combinations of the elements of $S$, i.e., $\cone(S)=\R_+S$.  Moreover, for a subsemigroup $C$ of $\Z^n$, we denote the smallest subgroup of $\Z^n$ containing $C$ by $\Z C$.  From these definitions, we now describe the edge ring of a graph $G$.

\begin{definition}
Let $G$ be a signed graph.  The {\em edge ring} of $G$ is the affine semigroup ring generated by $\rho(E(G))$, i.e.,
$$
k[G]\coloneqq k[\Z_+\rho(E(G))]=k[x^a]_{a\in \rho(E(G))}.
$$
\end{definition}

\subsection{Serre's Conditions}\label{sec:SerreConditions}

In this section, we recall Serre's conditions $R_\ell$ and $S_\ell$ for rings and some of the applications of these conditions to commutative algebra.  We also include a theorem of Hibi and Katth\"{a}n \cite[Theorem~3.2]{HibiKattan:2014}, based on a result of Vitulli \cite[Theorem~2.7]{Vitulli:2009} which characterizes the $R_1$ condition for semigroup rings.

\begin{definition}
A finitely generated module $M$ over a Noetherian ring $R$ satisfies {\em Serre's condition $S_{\ell}$} if $\mbox{depth} (M_p)\geq \min(\ell, \dim M_p)$ for all $p\in \mbox{Spec}\ R$.  A Noetherian ring $R$ satisfies {\em Serre's condition $R_{\ell}$} if $R_{\mathfrak{p}}$ is a regular local ring for all prime ideals $\mathfrak{p}$ in $R$ with $\dim R_{\mathfrak{p}} \leq \ell$
\end{definition}

Serre's conditions $R_\ell$ and $S_\ell$ characterize many interesting properties of rings, see, e.g., \cite{Grothendieck:1965,BrunsHerzog-CohenMacaulay,Serre:2000} for details.  For example, a ring is Cohen-Macaulay if and only if it satisfies $S_\ell$ for all $\ell$.  Moreover, a Noetherian ring is normal if and only if it satisfies $R_1$ and $S_2$, see, e.g., \cite[Theorem 5.8.6]{Grothendieck:1965}.  Additionally, Hochester shows that if $C$ is a normal semigroup and $k$ is a field, then the semigroup ring $k[C]$ is Cohen-Macaulay, see, e.g., {\cite[Theorem~6.3.5(a)]{BrunsHerzog-CohenMacaulay}}.

Serre's $R_\ell$ conditions can be combinatorially characterized for affine semigroup rings, see \cite[Theorem~2.7]{Vitulli:2009}.  In this paper, we focus on the $R_1$ condition.  We collect the characterization of this special case here:

\begin{proposition}[see {\cite[Proposition~3.2]{HibiKattan:2014}}]\label{prop:vitulli1}
Let $C$ be an affine semigroup, $k$ a field, and $k[C]$ the associated semigroup ring.  $k[C]$ satisfies Serre's $R_1$ condition if and only if for every facet $F$ of $\cone(C)$, we can find a supporting linear form $\sigma_F$ and corresponding hyperplane $\mathcal{H}_F$ which satisfies the following:
\begin{itemize}
\item The form $\sigma_F$ takes integral values on $\Z C$,
\item There exists $x\in C$ so that $\sigma_F(x)=1$, and
\item The groups $\Z(C \cap F)$ and $\Z C \cap \mathcal{H}_F$ are equal.
\end{itemize}
\end{proposition}

We recall that the {\em facets} of a cone are the proper faces with maximum dimension, and a {\em support form} of a facet $F$ is a function $\sigma_F(x)=\langle v^\ast,x\rangle$ which is nonnegative on the cone and zero on the facet.  Using this result, Hibi and Katth\"{a}n combinatorially characterize the Serre's $R_1$ condition for edge subrings of $k[x_1,\dots,x_n]$, see \cite[Theorem~2.1]{HibiKattan:2014}.  In this paper, we extend this characterization to the more difficult case of all quadratic-monomial generated subrings of the Laurent polynomial ring.

\section{Serre's \texorpdfstring{$R_1$}{R1} Condition for Signed Graphs}
\label{sec:Geometry}

In this section, we first characterize the facets of $\cone(\PG)$ in terms of subgraphs of the graph $G$. Then, we apply Proposition \ref{prop:vitulli1} to characterize the graphs $G$ for which the corresponding quadratic-monomial subrings of $k[x_1^{\pm 1},\dots,x_n^{\pm}]$ whose generators are of the form $(x_ix_j)^{\pm 1}$ satisfy $R_1$.  In order to study the facets, the first step is to compute the dimension of $\cone(\PG)$.

\subsection{Dimension of \texorpdfstring{$\cone(\mathcal{P}_G)$}{cone(PG)}}
\label{sec:Dimension}
In this section, we compute the dimension of the cone generated by the edge polytope $\PG$ by constructing a collection of independent hyperplanes which contain the polytope.  In particular, the codimension of $\cone(\PG)$ is the dimension of the space of hyperplanes of the form $\{x\in\mathbb{R}^n:\langle v^\ast,x\rangle=0\}$ containing the cone.  Here, $v^\ast$ is a dual vector in the dual vector space $(\R^n)^\ast$.  It is often useful to view $v^*$ as a set of vertex weights for $G$, i.e., if the hyperplane $\{w:\langle v^*,w\rangle =0\}$ contains $\rho(E)$, then, for all $ij\in E$, $(v^*)_i + (v^*)_j = 0$.  For notational convenience, we write dual vectors in terms of the standard dual basis, that is $v^*=\sum_i \lambda_i e_i^*$ where $\{e_i^*\}$ is the basis dual to $\{e_i\}$.  

Suppose $G$ is a signed graph and $H$ is a bipartite component of $H$.
Let $V(H)=L\cup R$ be the bipartition of $H$, then we write $e_L^*=\sum_{i\in L}e_i^*$ and $e_R^*=\sum_{j\in R}e_j^*$ for the dual characteristic vectors for $L$ and $R$.
With this notation, every edge $ij$ in $H$ has $\langle e_L^*,\rho(ij)\rangle =\langle e_R^*,\rho(ij)\rangle =\sgn(ij)$.
Hence, $\langle e_L^* - e_R^*,\rho(ij)\rangle =0$ for every edge in $G$.
If $G$ has multiple bipartite components, then we observe that the supports of the hyperplanes for distinct bipartite components are disjoint, so the hyperplanes are independent.

\begin{definition}
Let $G$ be a signed graph.  We write $\bicomp(G)$ for the number of bipartite components of $G$.
\end{definition}

Since the supports for the hyperplanes for disjoint bipartite components are disjoint and these hyperplanes contain $\PG$, we observe that there are at least $\bicomp(G)$ independent hyperplanes that contain $\PG$.

\begin{example}
Let $G$ be the graph in Figure \ref{fig:exampleC6}, i.e., $G$ has vertex set $\{1,2,3,4,5,6\}$ and edge set $\{+12,-23,+34,+45,-56,+16\}$.
This is a bipartite graph with bipartition $\{1,3,5\}$ and $\{2,4,6\}$.
Hence $\PG$ is contained in the hyperplane defined by $\langle (e_1^*+e_3^*+e_5^*)-(e_2^*+e_4^*+e_6^*),x\rangle =0$.
\end{example}

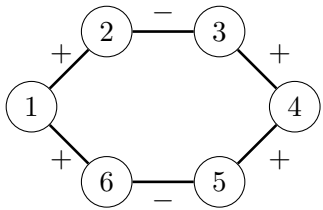
\begin{figure}[hbt]
	\centering
		\begin{tikzpicture}
		\node [draw,circle] (A) at (0,1) {1};
		\node [draw,circle] (B) at (1,2) {2};
		\node [draw,circle] (C) at (2.5,2) {3};
		\node [draw,circle] (D) at (3.5,1) {4};
		\node [draw,circle] (E) at (2.5,0) {5};
		\node [draw,circle] (F) at (1,0) {6};
		
		\node at (.4,1.7) {$+$};
		\node at (1.75,2.25) {$-$};
		\node at (3.3,1.7) {$+$};
		\node at (.4,.3) {$+$};
		\node at (1.75,-.25) {$-$};
		\node at (3.3,.3) {$+$};
		
		\draw[line width=1pt] (A) edge (B);
		\draw[line width=1pt] (B) edge (C);
		\draw[line width=1pt] (C) edge (D);
		\draw[line width=1pt] (D) edge (E);
		\draw[line width=1pt] (E) edge (F);
		\draw[line width=1pt] (F) edge (A);
		\end{tikzpicture}
		\caption[Example of a 4-dimensional polytope from a signed graph.]{The signed graph $G$ where $\PG$ is contained in a hyperplane due to the biparition.  The hyperplane corresponding to the bipartition is $\langle (e_1^*+e_3^*+e_5^*)-(e_2^*+e_4^*+e_6^*),x\rangle =0$, thus, $\dim\cone(\PG)=6-1=5$.
		 \label{fig:exampleC6}}
\end{figure}

\begin{lemma}\label{lem:BipHyper}
Let $G$ be a signed graph with bipartite components $G_1,G_2,\dots,G_k$, and associated dual vectors $e_{L_1}^*-e_{R_1}^*,\dots,e_{L_k}^*-e_{R_k}^*$.
Suppose $\PG$ is contained within the hyperplane defined by $\langle v^*,x\rangle =0$.  Then, $v^*$ is a linear combination of $e_{L_1}^*-e_{R_1}^*,\dots,e_{L_k}^*-e_{R_k}^*$.
\end{lemma}
\begin{proof}
The claim is trivial when $v^\ast=0$; we assume, therefore, that $v^\ast\not=0$. Suppose, first, that $G$ is a connected graph. Observe that, for all edges $ij$, $\langle v^\ast,\rho(ij)\rangle=\sgn(ij)\left((v^\ast)_i+(v^\ast)_j\right)=0$; it follows that, $(v^\ast)_i=-(v^\ast)_j$.  Since $v^\ast\not=0$, there is some $i$ such that $(v^\ast)_i=c\not=0$.  From connectivity and the observation that neighboring vertices have opposite signs, it follows that for all vertices $j$, $(v^\ast)_j=\pm c$.  Define $L$ to be the set of vertices whose weight is $c$ and $R$ to be the set of vertices whose weight is $-c$.  $L$ and $R$ form a partition of $G$ and they cannot contain any edges because the signs of the endpoints of an edge have opposite signs.  Therefore, $L$ and $R$ form the bipartition of $G$ and $v^* =  c(e_L^* - e_R^*)$.
For a disconnected graph, by applying this argument to each component, it follows that $v^*$ is a linear combination of these dual vectors.
\end{proof}

We have shown that the dual vectors $e_{L_i}^\ast-e_{R_i}^\ast$ form a basis for those hyperplanes containing $\cone(\PG)$.  Therefore, we can compute the dimension of $\cone(\PG)$ as follows:

\begin{observation}\label{thm:dimFormula}
Let $G$ be a signed graph on $n$ vertices, then  \[\dim \cone(\PG) = n-\bicomp(G).\]
\end{observation}

\begin{example}\label{example:firstexample}
Let $G$ be the graph from Example \ref{example:firstexample}.  Observe that $G$ is a connected bipartite graph with $n=6$ and $\bicomp(G)=1$.  Thus, by Observation \ref{thm:dimFormula}, $\dim\cone(\PG)= n-\bicomp(G)=5$.
\end{example}

\subsection{Serre's \texorpdfstring{$R_1$}{R1} Condition for \texorpdfstring{$k[G]$}{k[G]}}

In this section, we apply Proposition \ref{prop:vitulli1} to characterize the signed graphs $G$ for which $k[G]$ satisfies $R_1$.  Since the characterization in Proposition \ref{prop:vitulli1} is in terms of facets of $\cone(\PG)$, we begin by characterizing the facets of $\cone(\PG)$.  By the basic theory of convex cones, see, e.g., \cite{CoxLittleSchenck-ToricVarieties}, every proper face of $\cone(\PG)$ is given by the intersection of the cone with a supporting hyperplane $\mathcal{H}$.  Moreover, $\mathcal{H}\cap\cone(\PG)$ is a cone whose generators are the generators of $\cone(\PG)$ in $\mathcal{H}$.  In particular, the set of edges $e\in G$ such that $\rho(e)\in\mathcal{H}$ form a subgraph $H$ of $G$ such that $\mathcal{H}\cap\cone(\PG)=\cone(\mathcal{P}_H)$.  In this section, we characterize those subgraphs $H$ of $G$ such that $\cone(\mathcal{P}_H)$ is a facet of $\cone(\PG)$.  

We begin by noting the necessary, but not sufficient fact from Observation \ref{thm:dimFormula} that, for a subgraph $H$ to determine a facet of $\cone(\PG)$, $H$ must have exactly one more bipartite component than $G$, i.e., $\bicomp(H)=\bicomp(G)+1$.

\begin{definition}\label{def:facetsubgraphsigned}
Let $G$ be a signed graph.  A subgraph $H$ of $G$ is a {\em facet subgraph} if $H$ satisfies the following properties:
\begin{enumerate}
\item $H$ has exactly one more bipartite component than $G$ and
\item For any bipartite component $H'$ of $H$ which is not a component of $G$, there is a bipartition $H'=L\cup R$ such that every edge $e\in G\setminus H$ is of one of the following forms:
\begin{itemize}
\item $e$ is a positive edge incident to $L$, but not $R$ or 
\item $e$ is a negative edge incident to $R$, but not $L$. 
\end{itemize}
\end{enumerate}
\end{definition}

\begin{remark}
The characterization for these subgraphs in \cite{HibiKattan:2014} is based upon independent sets of vertices.  Since all graph edges in \cite{HibiKattan:2014} have positive sign and $G$ is connected, the independent subset corresponds to one of $L$ or $R$.  However, the notion of independence is not fine enough to distinguish the case of a signed graph because there are two types of incident edges.
\end{remark}

Our goal is to show that the facet subgraphs are exactly those subgraphs where $\cone(\mathcal{P}_H)$ is a facet of $\cone(\PG)$.  

\begin{proposition}\label{prop:signedConeFacets}
Let $G$ be a signed graph.  $F$ is a facet of $\cone(\PG)$ if and only if there exists a facet subgraph $H$ such that $\cone(\mathcal{P}_H)=F$.
\end{proposition}
\begin{proof}
Suppose first that $H$ is a facet subgraph.  Since $H$ has one additional bipartite component than $G$, we know that $\dim\cone(\mathcal{P}_H)=\dim\cone(\PG)-1$.  Therefore, it is enough to show that $\cone(\mathcal{P}_H)$ is a face of $\cone(\PG)$.  Let $H'$ be a bipartite component of $H$ which is not a component of $G$.  Let $H'=L\cup R$ be as in the definition of a facet subgraph, and consider $e_L^\ast-e_R^\ast$.  By the discussion in Section \ref{sec:Dimension}, we know that the hyperplane $\mathcal{H}=\{x\in\R^n:\langle e_L^\ast-e_R^\ast,x\rangle=0\}$ contains $\cone(\mathcal{P}_H)$.  On the other hand, for any edge $e\in G\setminus H$, we see, by analyzing the cases in the definition of a facet subgraph, that $\langle e_L^\ast-e_R^\ast,\rho(e)\rangle=1,2$.  Therefore, $\mathcal{H}$ is a supporting hyperplane, and the set of edges $e$ such that $\rho(e)\in\mathcal{H}$ is the same as the set of edges of $H$.  Hence $\cone(\mathcal{P}_H)=\mathcal{H}\cap\cone(\PG)$.

For the other direction, let $\mathcal{H}$ be a supporting hyperplane for $\cone(\PG)$ and $\mathcal{F}=\mathcal{H}\cap\cone(\PG)$ be a facet of $\cone(\PG)$.  Then, there is a dual vector $v^\ast$ such that $\mathcal{H}=\{x\in\R^n:\langle v^\ast,x\rangle=0\}$ and $\langle v^\ast,x\rangle\geq 0$ for all $x\in\cone(\PG)$.  Let $H$ be the subgraph of $G$ consisting of all edges $e\in G$ satisfying $\langle v^\ast,\rho(e)\rangle=0$.  By the theory of cones, we know that $\mathcal{F}=\cone(\mathcal{P}_H)$.  We show that $H$ is a facet subgraph.  Since $\dim\mathcal{F}=\dim\cone(\PG)-1$, it follows, from Observation \ref{thm:dimFormula}, that $H$ has exactly one more bipartite component than $G$.  Therefore, it is enough to check the second condition.  Let $H'$ be any bipartite component of $H$ which is not a component of $G$.  Moreover, let $H'=L\cup R$ be any bipartition of $H'$.

Let $G_1,\dots,G_k$ be the bipartite components of $G$ with corresponding bipartite characteristic vectors $e^\ast_{L_1}-e^\ast_{R_1},\dots e^\ast_{L_k}-e^\ast_{R_k}$.  From Section \ref{sec:Dimension}, we know that these vectors are independent and their corresponding hyperplanes contain $\cone(\PG)$, so they also contain $\cone(\mathcal{P}_H)$.  Now, consider $e_L^\ast-e_R^\ast$ corresponding to $H'$.  By Section \ref{sec:Dimension}, we see that the hyperplane corresponding to this vector contains $\cone(\mathcal{P}_H)$.  We can see that this vector is not a linear combination of the bipartite characteristic vectors above as follows: Since $H'$ is not a component of $G$, there is some edge $e\in G\setminus H$ incident to $H'$.  Since the hyperplanes corresponding to the dual vectors $e^\ast_{L_i}-e^\ast_{R_i}$ contain $\PG$ for all $i$, it follows that $\langle e^\ast_{L_i}-e^\ast_{R_i},\rho(e)\rangle=0$.  On the other hand, if only one endpoint of $e$ is incident to $H'$, then $\langle e_L^\ast-e_R^\ast,\rho(e)\rangle=\pm 1\not=0$, or if both endpoints are in $L$ or both are in $R$, then $\langle e_L^\ast-e_R^\ast,\rho(e)\rangle=\pm 2\not=0$.  In each of these cases, we see that since $\langle e_L^\ast-e_R^\ast,\rho(e)\rangle\not=0$, $e_L^\ast-e_R^\ast$ cannot be a combination of the bipartite characteristic vectors as they would give a value of $0$.  We now show that the remaining case, i.e., where one endpoint is in $L$ and the other is in $R$ is impossible.  Suppose that $e=ij$ with $i\in L$ and $j\in R$.  By replicating the proof of Lemma \ref{lem:BipHyper}, we see that $(v^\ast)_i=-(v^\ast)_j$, so $\langle v^\ast,\rho(e)\rangle=(v^\ast)_i+(v^\ast)_j=0$.  Therefore, $e\in H$, a contradiction.

By considering dimensions, we see that the intersection of the hyperplanes determined by $e^\ast_{L_1}-e^\ast_{R_1},\dots e^\ast_{L_k}-e^\ast_{R_k}$ and $e_L^\ast-e_R^\ast$ must be the the subspace containing $\cone(\mathcal{P}_H)$.  Therefore, $v^\ast$ can be written as a linear combination of these vectors.  Moreover, the coefficient of $e_L^\ast-e_R^\ast$ must be nonzero since otherwise $v^\ast$ would contain all of $\cone(\PG)$.  By reversing $L$ and $R$, if necessary, we may assume that the coefficient of $e_L^\ast-e_R^\ast$ is positive.  Therefore, for all $x$ in the span of $\PG$, $\langle v^\ast,x\rangle> 0$ if and only if $\langle e_L^\ast-e_R^\ast,x\rangle>0$.

We now check the second property for a facet subgraph.  Suppose that $e\in G\setminus H$.  Therefore, $\langle v^\ast,\rho(e)\rangle>0$, so $\langle e_L^\ast-e_R^\ast,\rho(e)\rangle>0$.  If $\sgn(e)=+1$, then either both endpoints of $e$ are in $L$, or one is in $L$ and the other is not in $H'$.  Similarly, when $\sgn(e)=-1$, either both endpoints of $e$ are in $R$ or one is in $R$ and the other is not in $H'$.  Therefore, since $H'$ was arbitrary, the second condition in the definition of a facet subgraph holds, and $H$ is a facet subgraph.
\end{proof}

In a less formal way, we can describe the construction of facet subgraphs from $G$:

\begin{observation}\label{obs:bipartitecount}
There are two ways to increase the number of bipartite components in the transformation from $G$ to the subgraph $H$: either we remove enough edges from a bipartite component of $G$ so that it splits into two components or we remove enough edges from a non-bipartite component of $G$ so that it splits into a bipartite component and any number of non-bipartite components.  The second condition of the definition of a facet subgraph implies that exactly one component of $G$ is changed in a facet subgraph since every removed edge must be incident to a component of $H$.  Moreover, if $G_1$ is a component of $G$ that splits into two nonempty components $H_1$ and $H_2$, with $H_1$ bipartite, then $G_1$ is bipartite if and only if $H_2$ is bipartite.
\end{observation}

We observe that the support form $\langle e_L^\ast-e_R^\ast,x\rangle$ is integral on $\Z \rho(E)$ since its value on every edge is $0$, $1$, or $2$.  Let $H'$ be as in the definition of a facet subgraph.  If there exists an edge $e\in G\setminus H$ such that exactly one endpoint of $e$ is incident to $H'$, then $\langle e_L^\ast-e_R^\ast,\rho(e)\rangle=1$, as needed in Proposition \ref{prop:vitulli1}.  On the other hand, if every edge $e\in G\setminus H$ has both endpoints in $H'$, then the values of the support form $\langle e_L^\ast-e_R^\ast,x\rangle$ are $0$ and $2$, so the support form $\frac{1}{2}\langle e_L^\ast-e_R^\ast,x\rangle$ satisfies the first two conditions of Proposition \ref{prop:vitulli1}.  We now consider the third condition from Proposition \ref{prop:vitulli1}.  We begin with the following lemma, which characterizes the integer latices used in the third condition of the proposition.

\begin{lemma}\label{lem:signedLattice}
Suppose that $G$ is a connected signed graph with $n$ vertices.  Let $a\in\Z^n$.  Then $a\in\Z \rho(E(G))$ if and only if the one of the following conditions hold:
\begin{itemize}
\item $G$ is not bipartite and 
\begin{equation}\label{eq:firstSignedLattice}
\sum_{i\in G} a_i \in 2\Z.
\end{equation}
\item $G=L\cup R$ is a bipartition of $G$ and 
\begin{equation}\label{eq:secondSignedLattice}
\sum_{i\in L} a_i = \sum_{j\in R} a_j.
\end{equation}
\end{itemize}
\end{lemma}

\begin{proof}
Assume first that $G$ is not bipartite.  Since, for every edge $e$ of $G$, $\sum_{i\in V(G)}\rho(e)_i=\pm 2$, the forward direction of the first case is proved.  For the other direction, let $i$ and $j$ be any pair of vertices in $G$.  Since $G$ is connected and $G$ is not bipartite, there are walks $\{i=i_0,i_1,\dots,i_m=j\}$ and $\{i=j_0,j_1,\dots,j_k=j\}$ of odd and even length, respectively.  By choosing integer weights $b_\ell=\pm 1$ along the odd walk, we can make $b_\ell\sgn(i_{\ell-1}i_\ell)$ alternate in sign with $b_1\sgn(ii_1)=1=b_m\sgn(i_{m-1}j)$.  In this case, $\sum b_\ell\rho(i_{\ell-1}i_\ell)=e_i+e_j$.  Similarly, we can choose integer weights $c_\ell=\pm 1$ along the even walk, we can make $c_\ell\sgn(j_{\ell-1}j_\ell)$ alternate in sign with $c_1\sgn(ij_1)=1$ and $c_k\sgn(j_{k-1}j)=-1$.  In this case, $\sum b_\ell\rho(i_{\ell-1}i_\ell)=e_i-e_j$.  Since every vector $a\in\Z^n$ whose coordinates sum to an even number is an integral linear combination of $e_i+e_j$ and $e_i-e_j$, for various $i$ and $j$, the equivalence holds.

Assume now that $G$ is bipartite.  To prove the forward direction, since every edge $e\in G$ is incident to a vertex in both $L$ and $R$, the given equality holds for those edges, and, hence, for their linear combinations.  For the other direction, let $i\in L$ and $j\in R$ be any pair of vertices.  Then, there is an odd length path $\{i=i_0,i_1,\dots,i_m=j\}$ in $G$ from $i$ to $j$.  By choosing integer weights $b_\ell=\pm 1$ along this walk, we can make $b_\ell\sgn(i_{\ell-1}i_\ell)$ alternate in sign with $b_1\sgn(ii_1)=1=b_m\sgn(i_{m-1}j)$.  In this case, $\sum b_\ell\rho(i_{\ell-1}i_\ell)=e_i+e_j$.  Since every vector $a\in\Z^n$ which satisfies the given equality can be written as an integral linear combination of $e_i+e_j$, for various $i$ and $j$, the equivalence holds.
\end{proof}

We observe that for a graph with multiple components, the latices in the third condition in Proposition \ref{prop:vitulli1} are the direct sums of the lattices for each component.  Therefore, Lemma \ref{lem:signedLattice} can be applied to a graph component-by-component.  We use this approach to prove the condition for a signed graph to satisfy Serre's $R_1$ condition. 

\begin{theorem}\label{thm:signedSerres}
Let $G$ be a signed graph.
$k[G]$ satisfies Serre's $R_1$ condition if and only if every facet subgraph $H$ of $G$ satisfies the following inequality:
\[\comp(H) \leq \comp(G) +1.\]
\end{theorem}

\begin{proof}
We use the conditions in Proposition \ref{prop:vitulli1} to justify Serre's $R_1$ condition.  By the discussion above, we have seen that for any facet subgraph $H$ of $G$, there is a linear form satisfying the first two conditions of Proposition \ref{prop:vitulli1}, so we focus on the third condition.

Suppose, first, that $\comp(H)>\comp(G) +1$.  In this case, there are more non-bipartite components in $H$ than in $G$.  Hence, there are two non-bipartite components $H_1$ and $H_2$ in $H$ which are in the same (non-bipartite) component of $G$.  Let $i\in H_1$ and $j\in H_2$.  By Lemma \ref{lem:signedLattice}, we know that $e_i+e_j\in\Z\rho(E(G))$.  Moreover, $e_i+e_j$ is in the supporting hyperplane for $F$ since $\langle e_L^\ast-e_R^\ast,e_i+e_j\rangle=0$ as $i,j\not\in L\cup R$ from Proposition \ref{prop:signedConeFacets}.  However, $e_i+e_j$ is not in $\Z\rho(E(H))$ since $i$ and $j$ are in distinct components.  More precisely, $e_i$ and $e_j$ would each need to be in $\Z\rho(E(H))$, which contradicts Lemma \ref{lem:signedLattice}.  Hence, $G$ fails the conditions in Proposition \ref{prop:vitulli1}, and, thus, does not satisfy Serre's $R_1$ condition.

Suppose, now, that $\comp(H)\leq\comp(G)+1$.  Since $H$ has at least as many components as $G$, $H$ has either the same number of components of $G$ or one more component.  Let $H_1$ be a bipartite component of $H$ which is not a component of $G$, and let $G_1$ be the component of $G$ containing $H_1$.  By Observation \ref{obs:bipartitecount}, we know that only component $G_1$ of $G$ is changed in the construction of $H$.  Therefore, we reduce to the case where $G$ is connected, i.e., $G=G_1$ and $H$ is the corresponding subgraph of $G_1$.  In what follows, we assume that $G$ has $m$ vertices.

Since $H_1$ is bipartite, let $H_1=L_1\cup R_1$ be a bipartition of $H_1$.  We recall, from the proof of Proposition \ref{prop:signedConeFacets}, that the facet $\cone(\mathcal{P}_H)$ is formed by the intersection of $\cone(\PG)$ with the hyperplane $\mathcal{H}=\{x\in\mathbb{R}^m:\langle e_{L_1}^\ast-e_{R_1}^\ast,x\rangle=0\}$.  Let $a\in \Z\rho(E(G))\cap\mathcal{H}\subseteq\Z^m$.  In order to confirm the third condition of Proposition \ref{prop:vitulli1}, we must show that $a\in\Z\rho(E(H))$.  Since $a\in\mathcal{H}$, we know that 
\begin{equation}\label{eq:balancingR1}
\sum_{i\in L_1}a_i=\sum_{j\in R_1}a_j
\end{equation}
since the weights on $L_1$ and $R_1$ must balance for the inner product above to vanish.  By Lemma \ref{lem:signedLattice}, this equality implies that for the bipartite graph $H_1$, 
$$\sum_{i\in H_1}a_ie_i\in\Z\rho(E(H_1))\subseteq \Z\rho(E(H)).$$  
Therefore, our goal is to show that the remaining portion of $a$ is also in $\Z\rho(E(H))$, i.e., we must show that
\begin{equation}\label{eq:difference}
a':=a-\sum_{i\in H_1}a_ie_i\in\Z\rho(E(H_2)).
\end{equation}

By appealing to Observation \ref{obs:bipartitecount}, we consider three cases:  If $\comp(H)=\comp(G)$, then $G$ and $H_1$ have the same vertices and $a'$ is zero.  If $\comp(H)=\comp(G)+1$, then $H$ consists of two components, let $H_2$ be the complementary component to $H_1$.  We observe that the expression in Equation (\ref{eq:difference}) can be rewritten as
$$
a'=\sum_{j\in H_2}a_je_j.
$$
We now proceed by considering two cases, depending on whether $G$ is bipartite or not.

Assume, first, that $G$ is bipartite, then there is a bipartition $G=L\cup R$ so that $L_1=L\cap H_1$ and $R_1=R\cap H_1$.   Moreover, since $G$ is bipartite, it follows that $H_2$ is also bipartite.  Let $L_2=L\cap H_2$ and $R_2=R\cap H_2$ form a bipartition of $H_2$.  By Lemma \ref{lem:signedLattice}, since $G$ is bipartite, it follows that 
$$
\sum_{i\in L}a_i=\sum_{j\in R}a_j.
$$ 
By appealing to Equation (\ref{eq:balancingR1}), it follows that 
$$
\sum_{i\in L_2}a_i=\sum_{j\in R_2}a_j.
$$
By Lemma \ref{lem:signedLattice}, since $H_2$ is bipartite, $a'\in\Z\rho(E(H_2))$.

On the other hand, if $G$ is not bipartite, then $H_2$ is not bipartite because otherwise $H$ has too many bipartite components.  In this case, by Lemma \ref{lem:signedLattice},
$$
\sum_{j\in H_2}a_j
$$
is even.  By Lemma \ref{lem:signedLattice}, since $H_2$ is not bipartite, this implies that $a'\in\Z\rho(E(H_2))$.

Since, in either case, we find that $\Z\rho(E(G))\cap\mathcal{H}=\Z\rho(E(H_1))\oplus\Z\rho(E(H_2))$ and the third condition of Proposition \ref{prop:vitulli1} holds.  Therefore, $k[G]$ satisfies the conditions in Proposition \ref{prop:vitulli1}, and, thus, satisfies Serre's $R_1$ condition.
\end{proof}

\section{Serre's \texorpdfstring{$R_1$}{R1} Condition for Mixed Signed, Directed Graphs}
\label{sec:mixedSerres}

In Section \ref{sec:Geometry}, we provided a combinatorial characterization of Serre's $R_1$ condition for signed graphs.  This provides a characterization for all quadratic-monomial generated rings where the generators are of the form $x_ix_j$ or $x_i^{-1}x_j^{-1}$.  This section contains the main results of this paper, where we extend the characterization to all quadratic-monomial generated domains, i.e., we allow generators of the form $x_i^{-1}x_j$.  We follow the approach of \cite{LipmanBurr:2016} for reducing the general case to the case of signed graphs.  We begin by briefly reviewing this reduction.

\subsection{Mixed Signed, Directed Graphs}

In \cite{LipmanBurr:2016}, we observe that the combinatorial object that corresponds to a monomial of the form  $x_i^{-1}x_j$ is a directed edge.  Therefore, we define mixed signed, directed graphs as follows:

\begin{definition}\label{def:MixedGraphs}
A {\em mixed signed, directed graph} is a pair $G=(V,E)$ of {\em vertices}, $V$, and {\em edges}, $E$, where $E$ consists of a set of signed edges and {\em directed edges} between distinct vertex pairs.
As for signed graphs, we denote positive and negative edges between $i$ and $j$ as $+ij$ and $-ij$, respectively.
A directed edge from $i$ to $j$ is denoted $(i,j)$.
\end{definition}

Note that, for any vertex $i$, there may be positive and negative loops at $i$ denoted $+ii$ and $-ii$ respectively, but not directed loops (since directed loops correspond to the indentity).
Also, for a pair $i$ and $j$ of distinct vertices, any subset of the four possible edges $+ij,-ij,(i,j)$, and $(j,i)$ can be edges in $G$.  We now recall the definitions pertaining to edge rings in this case:

\begin{definition}\label{def:MixedPolytopes}
Let $G$ be a mixed signed, directed graph with $n$ vertices, possibly with loops, and multiple edges.
Define $\rho:E(G)\rightarrow \R^n$ as $\rho(e)=\sgn(e)(e_i + e_j) \in \R^n$ when $e=\sgn(ij)ij$ is a signed edge of the graph and as $\rho(e)=e_j-e_i$ when $e=(i,j)$ is a directed edge of the graph.
\end{definition}

The edge polytope and corresponding semigroups and edge rings are defined in analogously to the definitions in Section \ref{sec:SemiGPs}.  In \cite{LipmanBurr:2016}, we showed how to construct a signed graph from a mixed signed, directed graph as follows:

\begin{definition}\label{def:AugmentedGraph}
Let $G$ be a mixed signed, directed graph.
The {\em augmented signed graph} $\widetilde{G}$ of $G$ is a signed graph where each directed edge $(i,j)$ in $G$ is replaced by a vertex $t_{(i,j)}$ and a pair of edges $-it_{(i,j)}$ and $+t_{(i,j)}j$.  The new vertex $t_{(i,j)}$, adjacent to only $i$ and $j$, is called an {\em artificial vertex}.
\end{definition}

Suppose that $G$ is a mixed signed, directed graph with $n$ vertices and $m$ artificial vertices.  We use the set $\{\widetilde{e}_1,\dots,\widetilde{e}_n,\widetilde{e}_{t_1},\dots,\widetilde{e}_{t_m}\}$ as the basis of the codomain of the map $\rho:E(\widetilde{G})\rightarrow\mathbb{R}^{n+m}$ to distinguish it from the basis $\{e_1,\dots,e_n\}$ for the codomain of $\rho:E(G)\rightarrow\R^n$.  Using this definition, we extend the definitions for components and bipartite components to mixed signed directed graphs.

\begin{definition}\label{def:MixedComponent}
Let $G$ be a mixed signed, directed graph.  A subgraph $H$ is a {\em component} or {\em bipartite component} of $G$ if $\widetilde{H}$ is a component or bipartite component of the augmented signed graph $\widetilde{G}$, respectively.
\end{definition}

We note that $H$ is a component of $G$ if and only if the underlying (undirected and unsigned) graph of $H$ is a component of the underlying graph of $G$.  However, this equivalence cannot be generalized to bipartite components due to artificial vertices. 

We study the properties of a mixed signed, directed graph by studying the corresponding properties on its augmented signed graph.  For example, if one considers only integral edge weights case in the proof of \cite[Lemma 6]{LipmanBurr:2016}, we achieve the following result:

\begin{corollary}\label{cor:FF:Equality}
Let $G$ be a mixed signed, directed graph with augmented signed graph $\widetilde{G}$.  Consider $k[G]$ as a subring of $k[x_1^{\pm 1},\dots,x_n^{\pm 1},t_1^{\pm 1},\dots,t_m^{\pm 1}]$.  Let $k(G)$ denote the field of fractions of $k[G]$.
Then $k(\widetilde{G})\cap k(x_1,\dots,x_n) = k(G)$.
\end{corollary}

\subsection{Serre's \texorpdfstring{$R_1$}{R1} condition for \texorpdfstring{$k[G]$}{k[G]}}

In this section, we generalize the results from Section \ref{sec:Geometry} to apply to mixed signed, directed graphs.  In particular, this allows us to provide conditions for which quadratic-monomial generated rings have Serre's $R_1$ condition in terms of structural properties of the graph $G$.  Throughout this section, we let $\pi:\widetilde{G}\rightarrow G$ be the projection map that ignores the artificial vertices.

\begin{corollary}\label{cor:mixedBipHyper}
Let $G$ be a mixed signed, directed graph and suppose that $\PG$ is contained within the hyperplane defined by $\langle v^*,x\rangle =0$.  Let $\widetilde{G}$ be the augmented signed graph for $G$.  Suppose that $\widetilde{G}$ has bipartite components $\widetilde{G}_1,\dots,\widetilde{G_r}$ in $\widetilde{G}$ with associated dual vectors $\widetilde{e}_{L_1}^*-\widetilde{e}_{R_1}^*,\dots,\widetilde{e}_{L_r}^*-\widetilde{e}_{R_r}^*$.  Then $v^\ast$ is a linear combination of $e^\ast_{\pi(L_1)}-e^\ast_{\pi(R_1)},\dots,e^\ast_{\pi(L_r)}-e^\ast_{\pi(R_r)}$.
\end{corollary}

\begin{proof}
Observe that for each directed edge $(i,j)\in E(G)$, $(v^\ast)_i=(v^\ast)_j$ since $\rho(i,j)=e_j-e_i$.  Let $\widetilde{v}^\ast$ be the dual vector 
$$
\widetilde{v}^\ast=\sum_{i\in G} (v^\ast)_i\widetilde{e}_i^\ast-\sum_{(i,j)\in E(G)}(v^\ast)_i\widetilde{e}_{t_{(i,j)}}^\ast.
$$
We observe that, by construction, for non-artificial vertices $i$, $(\widetilde{v}^\ast)_i=(v^\ast)_i$.  Moreover, the hyperplane defined by $\langle \widetilde{v}^*,y\rangle =0$ contains $\mathcal{P}_{\widetilde{G}}$ since every edge in both $G$ and $\widetilde{G}$ is in the hyperplane by the definition of $v^\ast$, and, for every edge with an artificial vertex as an endpoint, the values of $\widetilde{v}^\ast$ cancel at the endpoints.  By Lemma \ref{lem:BipHyper}, we know that $\widetilde{v}^\ast$ is a linear combination of the dual vectors $\widetilde{e}_{L_1}^*-\widetilde{e}_{R_1}^*,\dots,\widetilde{e}_{L_r}^*-\widetilde{e}_{R_r}^*$.  Therefore, by restricting our attention to non-artificial vertices, the result follows.
\end{proof}

Since our interest is in facets of $\cone(\PG)$, we must compute the dimension of this cone.  In particular, using Corollary \ref{cor:mixedBipHyper}, we get the following formula:

\begin{proposition}\label{prop:mixedConeDimension}
Let $G$ be a mixed signed, directed graph on $n$ vertices, then
\[\dim \cone(\PG) = n -\bicomp(G).\]
\end{proposition}

\begin{proof}
The proof mirrors the computation for Observation \ref{thm:dimFormula}.  We observe that the dual vectors in Corollary  \ref{cor:mixedBipHyper} are spanning; moreover, they are independent because the dual vectors of the form $e^\ast_{\pi(L)}-e^\ast_{\pi(R)}$ have disjoint support.  Therefore, the vectors in Corollary  \ref{cor:mixedBipHyper} form a basis for the set of hyperplanes containing $\cone(\PG)$.
\end{proof}

We now extend the definition of facet subgraphs to mixed signed, directed graphs.

\begin{definition}\label{def:facetsubgraphdirected}
Let $G$ be a mixed signed, directed graph.  A subgraph $H$ of $G$ is a {\em facet subgraph} if $H$ satisfies the following properties:
\begin{enumerate}
\item $H$ has exactly one more bipartite component than $G$ and
\item For any bipartite component $H'$ of $H$ which is not a component of $G$, there is a bipartition $\widetilde{H}'=\widetilde{L}\cup \widetilde{R}$ such that every edge $e\in G\setminus H$ is one of the following forms:  Let $L=\pi(\widetilde{L})$ and $\pi(\widetilde{R})$, i.e., $L$ is the set of nonartificial vertices of $\widetilde{L}$ in $H'$ and similarly for $R$.
\begin{itemize}
\item $e$ is a positive edge incident to $L$, but not $R$, 
\item $e$ is a negative edge incident to $R$, but not $L$,
\item $e=(i,j)$ is a directed edge such that $j\in L$, but $i\not\in L$, or
\item $e=(i,j)$ is a directed edge such that $i\in R$, but $j\not\in R$.
\end{itemize}
\end{enumerate}
\end{definition}
Observe that, due to artificial vertices, $L$ and $R$ do not form a bipartition of $H'$.  In particular, both endpoints of directed edges are in the same set.  As in the case of signed graphs, facet subgraphs characterize the facets of $\cone(\PG)$.

\begin{proposition}\label{prop:mixedConeFacets}
Let $G$ be a mixed signed, directed graph.  $F$ is a facet of $\cone(\PG)$ iff there exists a facet subgraph $H$ such that $\cone(\mathcal{P}_H)=F$.
\end{proposition}
\begin{proof}
The proof is similar to the proof of Proposition \ref{prop:signedConeFacets}, so we leave most of the details to the interested reader.  If $H$ is a facet subgraph, then it has a bipartite component $H'$.  Let $H'=L\cup R$ be as in the definition of a facet subgraph.  In this case, the hyperplane $\mathcal{H}=\{x\in\mathbb{R}^n:\langle e^\ast_L-e^\ast_R,x\rangle=0\}$ is a supporting hyperplane for $\cone(\PG)$ by Corollary \ref{cor:mixedBipHyper}, $\cone(\mathcal{P}_H)=\mathcal{H}\cap\cone(\PG)$, and for any edge $e\in G\setminus H$, $\langle e^\ast_L-e^\ast_R,\rho(e)\rangle=1,2$.

For the other direction, let $\mathcal{H}$ be a supporting hyperplane for $\cone(\PG)$ and $\mathcal{F}=\mathcal{H}\cap\cone(\PG)$ be a facet of $\cone(\PG)$.  Then, there is a dual vector $v^\ast$ such that $\mathcal{H}=\{x\in\mathbb{R}^n:\langle v^\ast,x\rangle=0\}$ and $\langle v^\ast,x\rangle\geq 0$ for all $x\in\cone(\PG)$.  Let $H$ be the subgraph of $G$ consisting of those edges $e\in G$ so that $\rho(e)\in\mathcal{H}$.  Then, by following the proof of Proposition \ref{prop:signedConeFacets}, it follows that, by reversing $R$ and $L$, if necessary, for all $x$ in the span of $\PG$, $\langle v^\ast,x\rangle>0$ if and only if $\langle e^\ast_L-e^\ast_R,x\rangle >0$.  By a case-by-case analysis, we conclude that the only possible edges in $G\setminus H$ are the ones of the form above.
\end{proof}

We observe that if $H$ is a facet subgraph and $H'$ is a bipartite subgraph of $H$ which is not a component of $G$, then let $H'=L\cup R$ be as in the definition of a facet subgraph.  Then, for each $e\in G$, $\langle e^\ast_L-e^\ast_R,\rho(e)\rangle=0,1,2$ and  every edge $e\in G\setminus H$ has a positive value.  If there is an edge taking on the value $1$, then $\langle e^\ast_L-e^\ast_R,x\rangle$ is the support form needed in Proposition \ref{prop:vitulli1}, and, otherwise, $\frac{1}{2}\langle e^\ast_L-e^\ast_R,x\rangle$ is the desired support form.  Therefore, as above, we focus on the third condition in Proposition \ref{prop:vitulli1}.

\begin{theorem}\label{thm:mixedSerres}
Let $G$ be a mixed signed, directed graph.
$k[G]$ satisfies Serre's $R_1$ condition if and only if every facet subgraph $H$ of $G$ satisfies the following inequality:
\[\comp(H)\leq \comp(G) +1.\]
\end{theorem}

\begin{proof}
The proof is similar to the proof of Theorem \ref{thm:signedSerres}, so we leave some of the details to the interested reader.  We have already seen that for any facet subgraph $H$ of $G$, there is a linear form satisfying the first two conditions of Proposition \ref{prop:vitulli1}, so we focus on the third condition.

Suppose that $\comp(H)>\comp(G)+1$ and let $H_1$ and $H_2$ be two non-bipartite components in the same (non-bipartite) component $G_1$ of $G$.  Let $i\in H_1$ and $j\in H_2$.  Then, $\widetilde{e}_i+\widetilde{e}_j\in\Z\rho(E(\widetilde{G}))$ by Lemma \ref{lem:signedLattice}.  Then, by Corollary \ref{cor:FF:Equality}, since $\widetilde{e}_i+\widetilde{e}_j$ corresponds to the monomial $x^ix^j\in k(\widetilde{G})\cap k(x_1,\dots,x_n)$, $e_i+e_j\in \Z \rho(E(G))$.  For this element to be in $\Z\rho(E(H))$, both $e_i$ and $e_j$ would each need to be in $\Z\rho(E(H))$, but neither of these are in $\Z\rho(E(\widetilde{H}))$ by Lemma \ref{lem:signedLattice}.  Hence, $G$ fails the conditions in Proposition \ref{prop:vitulli1}, and, thus, does not satisfy Serre's $R_1$ condition.

Suppose now that $\comp(H)\leq\comp(G)+1$.  By following the proof of Theorem \ref{thm:signedSerres}, we choose $H_1$ and $G_1$ as above, i.e., $H_1$ is a bipartite component of $H$ which is not a component of $G$ and $G_1$ is the component of $G$ containing $H_1$.  As above, we restrict our attention to the case where $G$ consists of a single component, i.e., $G=G_1$.  In what follows, we assume that $G$ has $m$ vertices and $m'$ directed edges.

Since $H_1$ is bipartite, let $H_1=L_1\cup R_1$ be a partition of $H_1$ coming from a bipartition of the augmented graph $\widetilde{H}_1=\widetilde{L}_1\cup\widetilde{R}_1$.  We recall, from the proof of Proposition \ref{prop:mixedConeFacets}, that the facet $\cone(\mathcal{P}_H)$ is formed by the intersection of $\cone(\PG)$ with the hyperplane $\mathcal{H}=\{x\in\mathbb{R}^m:\langle e^\ast_{L_1}-e^\ast_{R_1},x\rangle=0\}$.

Let $a\in\Z\rho(E(G))\cap\mathcal{H}\subseteq\Z^m$, and define $\widetilde{a}\in\Z^{m+m'}$ where $(\widetilde{a})_i=a_i$ for $i\in G$ and $(\widetilde{a})_{t_i}=0$ for artificial vertices.  By Corollary \ref{cor:FF:Equality}, $\widetilde{a}\in\Z\rho(E(\widetilde{G}))$ since $a$ and $\widetilde{a}$ correspond to the same element in $k(G)$.  In order to confirm the third condition of Proposition \ref{prop:vitulli1}, we must show that $a\in\Z\rho(E(H))$, or, equivalently, by Corollary \ref{cor:FF:Equality}, that $\widetilde{a}\in\Z\rho(E(\widetilde{H}))$.  Since $a\in\mathcal{H}$, we know that 
\begin{equation}\label{eq:balancingR1:mixed}
\sum_{i\in L_1}a_i=\sum_{j\in R_1}a_j
\end{equation}
since the weights on $L_1$ and $R_1$ must balance for the inner product above to vanish.  Moreover, since all artificial vertices have weight zero, Equation (\ref{eq:balancingR1:mixed}) implies that 
\begin{equation*}
\sum_{i\in \widetilde{L}_1}\widetilde{a}_i=\sum_{j\in \widetilde{R}_1}\widetilde{a}_j.
\end{equation*}
Since $\widetilde{H}_1$ is bipartite, by Lemma \ref{lem:signedLattice}, 
$$\sum_{i\in \widetilde{H}_1}\widetilde{a}_i\widetilde{e}_i\in\Z\rho(E(\widetilde{H}_1)).$$
Since the weights on the artificial vertices are zero, by Corollary \ref{cor:FF:Equality}, it follows that 
$$\sum_{i\in H_1}a_ie_i\in\Z\rho(E(H_1))\subseteq\Z\rho(E(H)).$$
Therefore, our goal is to show that the remaining part of $a$ is also in $\Z\rho(E(H))$, i.e., we must show that 
\begin{equation}\label{eq:mixed:difference}
a':=a-\sum_{i\in H_1}a_ie_i\in\Z\rho(E(H_2)).
\end{equation}
We define $\widetilde{a}'$ by extending $a'$ to $\Z\rho(E(G))$ as above.

By adapting Observation \ref{obs:bipartitecount} to the mixed signed, directed case, we consider three cases: If $\comp(H)=\comp(G)$, then $G$ and $H_1$ have the same vertices and $a'$ is zero.  If $\comp(H)=\comp(G)+1$, then $H$ consists of two components, let $H_2$ be the complementary component to $H_1$.  We observe that the expression in Equation (\ref{eq:mixed:difference}) can be rewritten as
$$
a'=\sum_{j\in H_2}a_je_j.
$$
We now proceed by considering two cases, depending on whether $G$ is bipartite or not.

Assume, first, that $G_1$ is bipartite, then there is a partition $G=L\cup R$ coming from a bipartition of the augmented graph $\widetilde{G}=\widetilde{L}\cup\widetilde{R}$ so that $L_1=L\cap H_1$ and $R_1=R\cap H_1$.  Moreover, since $G$ is bipartite, it follows that $H_2$ is bipartite.  Let $L_2=L\cap H_2$ and $R_2=R\cap H_2$ be a partition of $H_2$ coming from a bipartition $\widetilde{L}_2=\widetilde{L}\cap\widetilde{H}_2$ and $\widetilde{R}_2=\widetilde{R}\cap\widetilde{H}_2$ of $\widetilde{H_2}$.  By Lemma \ref{lem:signedLattice}, since $\widetilde{G}$ is bipartite, it follows that
$$
\sum_{i\in\widetilde{L}}\widetilde{a}_i=\sum_{j\in\widetilde{R}}\widetilde{a}_j.
$$
By appealing to Equation (\ref{eq:balancingR1:mixed}), it follows that 
$$
\sum_{i\in\widetilde{L}_2}\widetilde{a}_i=\sum_{j\in\widetilde{R}_2}\widetilde{a}_j.
$$
Since $H_2$ is bipartite, by Lemma \ref{lem:signedLattice} we can conclude that $\widetilde{a}'\in\Z\rho(E(\widetilde{H}_2))$.  Since the weights on the artificial vertices are zero, by Corollary \ref{cor:FF:Equality}, it follows that $a'\in\Z\rho(E(H_2))$.

On the other hand, if $G$ is not bipartite, then $H_2$ is not bipartite since otherwise $H$ has too many bipartite components.  Since $\widetilde{G}$ is not bipartite, by Lemma \ref{lem:signedLattice},
$$
\sum_{j\in\widetilde{H}_2}\widetilde{a}_j
$$
is even.  By Lemma \ref{lem:signedLattice}, since $\widetilde{H}_2$ is not bipartite, this implies that $\widetilde{a}'\in\Z\rho(E(\widetilde{H}_2))$.  Since the weights on the artificial vertices are zero, by Corollary \ref{cor:FF:Equality}, it follows that $a'\in\Z\rho(E(H_2))$.

Since, in either case, we find that $\Z\rho(E(G))\cap\mathcal{H}=\Z\rho(E(H_1))\oplus\Z\rho(E(H_2))$ and the third condition of Proposition\ref{prop:vitulli1} holds.  Therefore, $k[G]$ satisfies the conditions in Proposition \ref{prop:vitulli1}, and, thus, satisfies Serre's $R_1$ condition.
\end{proof}

We end this section with a direct corollary of the theorem.

\begin{corollary}\label{cor:mixedSerres}
Let $G$ be a mixed signed, directed graph.
$k[G]$ satisfies Serre's $R_1$ condition if and only if $k[\widetilde{G}]$ satisfies Serre's $R_1$ condition.
\end{corollary}

\subsection{Example}

Recall, from Section \ref{sec:SerreConditions}, that (1) a ring is Cohen-Macaulay if and only if it satisfies $S_\ell$ for all $\ell$, (2) a Noetherian ring is normal if and only if it satisfies $R_1$ and $S_2$, and (3) normal semigroup rings over a field are Cohen-Macaulay.  Using these facts, we have the following direct observation which allows us to combinatorially construct rings which are not Cohen-Macaulay.

\begin{observation}\label{obs:NormAndCM}
Suppose that $G$ is a mixed signed, directed graph such that $k[G]$ is not normal, but satisfies $R_1$.  Then, $k[G]$ must fail $S_2$, so it cannot be Cohen-Macaulay.
\end{observation}

Before providing an example, we recall the following result from \cite{LipmanBurr:2016}, which characterizes, combinatorially which mixed signed, directed graphs have normal edge rings.

\begin{proposition}[see {\cite[Theorem 5]{LipmanBurr:2016}}]\label{prop:Normality}
Let $G$ be a mixed signed, directed graph.  $k[G]$ is normal if and only if $G$ satisfies the generalized odd cycle condition, i.e., for any two disjoint cycles, each with an odd number of signed edges, between them there is either no path or a generalized alternating path.
\end{proposition}

\begin{example}\label{ex:R1NotNormD} 
Consider the signed graph $G$ in Figure \ref{fig:R1NotNormD}, i.e., $G$ has vertex set $\{1,2,3,4\}$ and edge set $\{+11,-12,+13,(2,3),-24,+34,+44\}$, .  Se observe that $k[G]$ is not normal since there is not a generalized alternating path between the loops at $1$ and $4$.  On the other hand, we see that $k[G]$ satisfies Serre's $R_1$ condition as follows:

Suppose that $H$ is a facet subgraph of $G$ such that $\comp(H)>\comp(G)+1$.  Since $G$ has one component and $G$ is not bipartite, then $H$ must have exactly one bipartite component and at least two non-bipartite components.  Since there are only two cycles with an odd number of signed edges, i.e., the loops at $1$ and $4$, these two vertices must be in the two non-bipartite components.  If $2$ is the new bipartite component, then $1$ and $4$ are in the same component as $3$ connects them.  Similarly, if $3$ is the new bipartite component, then $1$ and $4$ are connected via $2$.  Therefore, both $2$ and $3$ must be in the bipartite component.  Since $2$ has negative incident edges, $2\in R$, and, similarly, $3\in L$.  However, the edge $(2,3)$ implies that both $2$ and $3$ are both in $\widetilde{L}$ or $\widetilde{R}$ in $\widetilde{H}$.  This contradicts the definition of a facet subgraph.  Thus, there are no facet subgraphs failing the inequality given in Theorem \ref{thm:mixedSerres}, and $k[G]$ satisfies Serre's $R_1$ condition.  Therefore, by observation \ref{obs:NormAndCM}, it follows that $k[G]$ must fail $S_2$, and, hence is not Cohen-Macaulay.
\end{example}

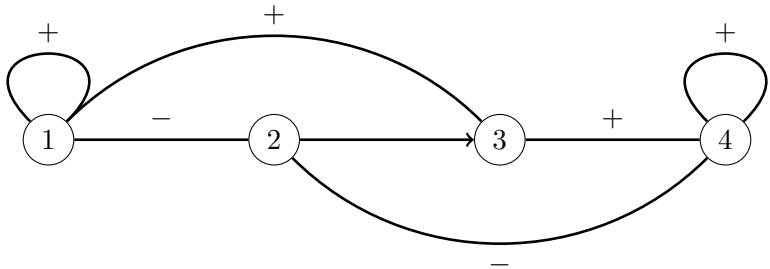
\begin{figure}[hbt]
	\centering
\begin{tikzpicture}
\node [draw,circle] (A) at (0,0) {1};
\node [draw,circle] (B) at (3,0) {2};
\node [draw,circle] (C) at (6,0) {3};
\node [draw,circle] (D) at (9,0) {4};
\draw[line width=1pt] (A) edge[out=45,in=135,looseness=9] node[above]{$+$} (A);
\draw[line width=1pt] (D) edge[out=45,in=135,looseness=9] node[above]{$+$} (D);
\draw[line width=1pt] (A) edge node[above]{$-$} (B);
\draw[->,line width=1pt] (B) -- (C);
\draw[line width=1pt] (C) edge node[above]{$+$} (D);
\draw[line width=1pt] (A) edge[out=45,in=135,looseness=1] node[above]{$+$} (C);
\draw[line width=1pt] (B) edge[out=-45,in=-135,looseness=1] node[below]{$-$} (D);
\end{tikzpicture}
	\caption{The edge ring for this mixed signed, directed graph satisfies $R_1$, but it fails to be normal.  Therefore, the edge ring fails condition $S_2$.  Thus, this is an example of a graph whose edge ring is not Cohen-Macaulay, by Observation \ref{obs:NormAndCM}.}
	\label{fig:R1NotNormD}
\end{figure}

\begin{remark}
The graph in Example \ref{ex:R1NotNormD} is a minimal example of an edge ring which satisfies $R_1$ and not $S_2$ in the following sense:  Any graph which has fewer vertices either fails $R_1$ or its edge ring is normal and so the edge ring satisfies $S_2$. 
\end{remark}

\section{Conclusion}\label{sec:Conclusion}

In this paper, we have extended the characterization of Serre's $R_1$ condition for edge rings presented by Hibi and Katth\"an \cite{HibiKattan:2014} to arbitrary mixed signed, directed graphs.  This results in an explicit complete characterization of Serre's $R_1$ condition for quadratic-monomial generated polynomial rings in terms of the combinatorial properties of the associated graphs.  The proofs presented in this paper are new and have additional subtleties than in \cite{HibiKattan:2014} due to the possibility of cancellation between terms.  This characterization of $R_1$ allows us to provide an example of a quadratically generated polynomial ring which is not Cohen-Macaulay.

\section*{Acknowledgements}\label{sec:Acknowledgements}
The authors would like to thank their colleague, Sean Sather-Wagstaff for helpful feedback on this project.

\bibliography{bibliographyNormalDomains}
\bibliographystyle{spmpsci}

\end{document}